\documentclass[reqno]{amsart}
\usepackage{graphicx, amsmath,amsfonts,amsthm,amssymb, paralist, fullpage}
\usepackage{tikz}
\usetikzlibrary{arrows,intersections}
\newtheorem{theorem}{Theorem}

\theoremstyle{definition}

\theoremstyle{remark}

\newcommand{\g}{\geqslant}

\newcommand{\RR}{\mathbb{R}}
\newcommand{\ZZ}{\mathbb{Z}}
\newcommand{\CC}{\mathbb{C}}

\newcommand{\les}{\leqslant}
\newcommand{\lesa}{\lesssim}

\newcommand{\mc}[1]{\mathcal{#1}}
\newcommand{\mb}[1]{\mathbf{#1}}

\newcommand{\ind}{\mathbbold{1}}
\newcommand{\ma}{\measuredangle}

\DeclareSymbolFont{bbold}{U}{bbold}{m}{n}
\DeclareSymbolFontAlphabet{\mathbbold}{bbold}

\DeclareMathOperator*{\supp}{supp}

\setdefaultenum{(i)}{(a)}{1}{A}

\begin{document}

\title{A note on bilinear wave-Schr\"odinger interactions}%
\author{Timothy Candy}%
\address[T.~Candy]{Department of Mathematics and Statistics, University of Otago, PO Box 56, Dunedin 9054, New Zealand}
\email{tcandy@maths.otago.ac.nz}

\thanks{Financial support by the Marsden Fund Council grant 19-UOO-142, and the German Research Foundation (DFG) through the CRC 1283 ``Taming uncertainty and profiting from
  randomness and low regularity in analysis, stochastics and their  applications'' is acknowledged.}


\begin{abstract}
We consider bilinear restriction estimates for wave-Schr\"odinger interactions and provided a sharp condition to ensure that the product belongs to $L^q_t L^r_x$ in the full bilinear range $\frac{2}{q} + \frac{d+1}{r} < d+1$, $1\les q, r \les 2$. Moreover, we give a counter-example which shows that the bilinear restriction estimate can fail, even in the transverse setting. This failure is closely related to the lack of curvature of the cone. Finally we mention extensions of these estimates to adapted function spaces. In particular we give a general transference type principle for $U^2$ type spaces that roughly implies that if an estimate holds for homogeneous solutions, then it also holds in $U^2$. This transference argument can be used to obtain bilinear and multilinear estimates in $U^2$ from the corresponding bounds for homogeneous solutions.
\end{abstract}
\maketitle

Let $u = e^{it|\nabla|} f$ be a free wave, and let $v = e^{it\Delta}g$  be a homogeneous solution to the Schr\"odinger equation. Our goal is to understand for which $1\les q, r \les \infty$ we have the bilinear estimate
        \begin{equation}\label{eqn:bilinear intro}
             \| u v\|_{L^q_t L^r_x(\RR^{1+d})} \lesa \| f \|_{L^2(\RR^d)} \| g \|_{L^2(\RR^d)}.
        \end{equation}
As a first step in this direction, assuming for instance that we have the support condition $\supp \widehat{f}, \supp \widehat{g} \subset \{ |\xi| \approx 1\}$, then for any
$\frac{2}{q_1} + \frac{d-1}{r_1} \les \frac{d-1}{2}$ with $(q_1, r_1, d ) \not = (2, \infty, 3)$, and any $\frac{2}{q_2} + \frac{d}{r_2} \les \frac{d}{2}$ with $(q_2, r_2, d)\not = (2, \infty, 2)$ we have the linear Strichartz estimates
   $$ \| u \|_{L^{q_1}_t L^{r_1}_x(\RR^{1+d})} \lesa \| f \|_{L^2(\RR^d)}, \qquad  \| v \|_{L^{q_2}_t L^{r_2}_x(\RR^{1+d})} \lesa \| g \|_{L^2(\RR^d)}.$$
Consequently an application of H\"older's inequality and a short computation shows that the bilinear estimate \eqref{eqn:bilinear intro} holds provided that
        \begin{equation}\label{eqn:bi range via stric}
          \frac{2}{q} + \frac{d}{r} \les d, \qquad \qquad  \frac{2}{q} + \frac{d-1}{r} \les d-1 + \frac{1}{d},\qquad \text{and} \qquad (q, d) \not = \Big(\frac{4}{3}, 2\Big), (1, 3).
        \end{equation}
The first condition in \eqref{eqn:bi range via stric} is stronger in the region $q\g 2$ and follows by simply placing $u\in L^\infty_t L^2_x$ and using the Strichartz estimate for $v$. Note that this explains the Schr\"odinger scaling of the first condition in \eqref{eqn:bi range via stric}. The second condition in \eqref{eqn:bi range via stric} dominates in the region $1\les q \les 2$, where we are forced to use the Strichartz estimates on both $u$ and $v$.

A natural question now arises, is it possible to improve on the conditions \eqref{eqn:bi range via stric}? This question is particularly relevant in applications to nonlinear PDE, where bilinear estimates such as \eqref{eqn:bi range via stric} with $q, r$ as small as possible, are extremely useful in controlling nonlinear interactions. Note that the wave-Schr\"odinger interactions occur naturally in important models, see for instance the Zakharov system \cite{zakharov_collapse_1972}. In the case of wave-wave interactions, it is possible to improve significantly on the range given by simply applying H\"older's inequality and the Strichartz estimate for the wave equation provided an additional transversality assumption is made.

\begin{theorem}[Bilinear restriction for wave {\cite{Wolff2001, Tataru2003}}] \label{thm:bilinear wave}
Let $d\g 2$ and $1\les q, r \les 2$ with $\frac{2}{q} + \frac{d+1}{r} < d+1$. If $f, g \in L^2(\RR^d)$ and $\omega, \omega' \in \mathbb{S}^{d-1}$ with\footnote{Here $\measuredangle(x,y) = ( 1 - \frac{x\cdot y}{|x| |y|})^\frac{1}{2}$ is the angle between $x, y\in \RR^d \setminus \{0\}$.} $\ma(\omega, \omega') \approx 1 $ and
        \begin{equation}\label{eqn:trans assump cone}
             \supp \widehat{f}  \subset \big\{ \xi \in \RR^d \,\,\big|\, |\xi| \approx 1, \ma(\xi, \omega) \ll 1 \big\}, \qquad \supp \widehat{g} \subset \big\{ \xi \in \RR^d \,\, \big|\, |\xi| \approx 1, \ma(\xi, \omega') \ll 1  \big\}
        \end{equation}
then
        $$ \| e^{it|\nabla|} f e^{it|\nabla|} g \|_{L^q_t L^r_x(\RR^{1+d}) } \lesa \| f \|_{L^2(\RR^d)} \| g \|_{L^2(\RR^d)}. $$
\end{theorem}

The first result beyond the linear Strichartz theory was obtained in \cite{Bourgain1995}. The endpoint and extension to more general frequency interactions is also known \cite{Tao2001b, Tataru2003, Temur2013}. The range for $(q, r)$ is sharp, and was originally conjectured by Klainerman-Machedon. Theorem \ref{thm:bilinear wave} is closely related to the restriction conjecture for the cone, as the free wave $e^{it |\nabla|} f$ is essentially the extension operator for the cone. In particular, bilinear estimates of the form \eqref{thm:bilinear wave} were originally used to obtain restriction estimates for the cone, see for instance \cite{Tao2000a}.

Theorem \ref{thm:bilinear wave} is truly a bilinear estimate as it relies crucially on the support assumption \eqref{eqn:trans assump cone}. This assumption implies that the two subsets of the cone, $\supp \mc{F}[e^{it|\nabla|}f] \subset \RR^{1+d}$ and $\supp  \mc{F}[e^{it|\nabla|} g]\subset \RR^{1+d}$, are transverse, where $\mc{F}$ denotes the space-time Fourier transform. Since the waves $e^{it|\nabla|} f$ and $e^{it|\nabla|} g$ propagate in the normal directions to these surfaces, the two waves can only interact strongly for short times. Thus we should expect the product $e^{it|\nabla|}f e^{it|\nabla|} g$ to decay faster than say $(e^{it|\nabla|}f)^2$.

If we apply the above discussion to the bilinear estimate \eqref{eqn:bilinear intro}, since the normal direction to the cone is $(1, -\frac{\xi}{|\xi|})$, and the normal direction to the paraboloid is $(1, 2\xi)$, we should expect to improve on the range \eqref{eqn:bi range via stric} obtained via the linear Strichartz estimates, by imposing a transversality condition of the form
        \begin{equation}\label{eqn:wave-schro trans}
            \Big| \frac{\xi}{|\xi|} + 2 \eta \Big| \gtrsim 1
        \end{equation}
for all $\xi \in \supp \widehat{f}$ and $\eta \in \widehat{g}$ (here $\widehat{f}$ denotes the spatial Fourier transform). Unfortunately, the simple transversality condition \eqref{eqn:wave-schro trans} does not suffice due to the lack of curvature of the cone along the surface of intersection
$$ \Sigma_{wave}(a,z)= \big\{ (\tau, \xi) \in \supp \mc{F}[e^{it|\nabla|} f] \,\, \big| \,\, (a, z) - (\tau, \xi) \in  \supp \mc{F}[e^{it\Delta} g]\big\}, \qquad (a, z) \in \RR^{1+d}.$$
In fact it is well known that for certain surfaces, transversality alone is not sufficient to obtain the full bilinear range, see for instance \cite{Lee2006} for the example of the hyperbolic paraboloid, and the related discussion in \cite{Bejenaru2017b, Candy2019b}. However, imposing a stronger support condition gives the following.

\begin{theorem}[Wave-Schr\"odinger bilinear restriction \cite{Candy2019b}]\label{thm:wave-schrodinger bilinear free solns}
Let $d\g 2$, $1\les q, r \les 2$, and $\frac{2}{q} + \frac{d+1}{r} < d+1$. Let $\xi_0, \eta_0 \in \RR^d$ such that
    \begin{equation}\label{eq:wave-schro general trans} \Big| \Big( \frac{\xi_0}{|\xi_0|} + 2 \eta_0 \Big) \cdot \frac{\xi_0}{|\xi_0|}\Big| \gtrsim \Big| \frac{\xi_0}{|\xi_0|} + 2 \eta_0 \Big| \end{equation}
and define $\lambda =  |\eta_0|$, and $\alpha = | \frac{\xi_0}{|\xi_0|} + 2 \eta_0 |$. If
	$$\supp \widehat{f} \subset \big\{ |\xi| \approx \lambda, \ma(\xi, \xi_0) \ll  \min\{1, \alpha \} \big\}, \qquad \supp \widehat{g} \subset \{ |\xi - \eta_0| \ll \alpha\}$$
then we have
    $$ \big\| e^{ it \Delta} f e^{ it |\nabla|} g \big\|_{L^q_t L^r_x(\RR^{1+d})} \lesa ( \min\{\alpha, \lambda, \alpha \lambda\})^{d+1 - \frac{d+1}{r} - \frac{2}{q}} \alpha^{\frac{1}{r}-1} \lambda^{\frac{1}{q} - \frac{1}{2}} \|f\|_{L^2_x} \| g \|_{L^2_x}. $$
\end{theorem}

Theorem \ref{thm:wave-schrodinger bilinear free solns} is a consequence of a bilinear restriction estimate for general phases obtained in \cite{Candy2019b}. The special case $q=r$ and $\alpha = \lambda = 1$ could also be deduced from \cite{Bejenaru2017b}. As the precise conditions in \cite{Candy2019b} are complicated, the derivation is slightly nontrival and we give the details below in Section \ref{sec:proof of bilinear}. The dependence on the parameters $\alpha$ and $\lambda$ is sharp, and this is particularly useful in applications to nonlinear PDE where $\alpha$ and $\lambda$ roughly correspond to a derivative loss/gain. Clearly, applying Sobolev embedding and interpolating with the trivial case $q=\infty$, $r=1$ can extend the range to $q, r \g 2$ and $\frac{2}{q} + \frac{d+1}{r}<d+1$. However the dependence on $\alpha$ and $\lambda$ would no longer be sharp (i.e. losses may occur). \\

The condition \eqref{eq:wave-schro general trans} is necessary to obtain the full bilinear range $\frac{2}{q} + \frac{d+1}{r} \les  d+1$.

\begin{theorem}[Transverse counter example]\label{thm:transverse counter}
Suppose that the estimate \eqref{eqn:bilinear intro} holds for all $f, g \in L^2(\RR^d)$ with\footnote{Here $e_j \in \RR^d$, $j=1, \dots, n$ denote the standard basis vectors.}
        $$ \supp \widehat{f} \subset \{ |\xi-e_1| \ll 1 \}, \qquad \supp \widehat{g} \subset \{ | 2 \xi + e_1 + e_2 | \ll 1 \}. $$
Then
    \begin{equation}\label{eqn:counter example range}
            \frac{2}{q} + \frac{d-1}{r} + \frac{1}{2r} \les d.
    \end{equation}
\end{theorem}

Note that if we let $\xi_0 = e_1$ and $\eta_0 =  - \frac{1}{2} e_1  - \frac{1}{2} e_2$, then $|\frac{\xi_0}{|\xi_0|} + 2 \eta_0| = 1$ but $(\frac{\xi_0}{|\xi_0|}+2\eta_0) \cdot \frac{\xi_0}{|\xi_0|} = 0$. In other words the transversality condition \eqref{eqn:wave-schro trans} holds, but the stronger condition \eqref{eq:wave-schro general trans} fails. The range \eqref{eqn:counter example range} is stronger than the bilinear range in Theorem \ref{thm:wave-schrodinger bilinear free solns} when $q$ is close to $1$ and $d\les 5$, see figure \ref{fig:range of q r}. If we drop the transversality condition completely, then a similar counter example can be used to prove the following.

\begin{figure}
\begin{tikzpicture}
    \draw[thick, <->] (0, 4.5) -- (0,0) -- (4.5, 0) ;
    \draw (0, -0.1) -- (0, 0.1) ;  
    \draw (4, -0.1) -- (4, 0.1);
    \draw (2, -0.1) -- (2, 0.1);
    \draw (-0.1, 0 ) -- ( 0.1, 0) ; 
    \draw (-0.1, 4) -- (0.1, 4);
    \draw (-0.1, 3) -- (0.1, 3);
    \node[align=left, below] at (0,-0.1) {$\frac{1}{2}$}; 
    \node[align=left, below] at (4,-0.1) {$1$};
    \node[align=left, below] at (4.6,0) {$\frac{1}{r}$};
    \node[align=left, below] at (2, -0.1) {$\frac{3}{4}$};
    \node[align=left, left] at (-0.1, 0) {$\frac{1}{2}$}; 
    \node[align=left, left] at (-0.1, 4) {$1$};
    \node[align=left, left] at (-0.1, 4.6) {$\frac{1}{q}$};
    \node[align=left, left] at (-0.1, 3) {$\frac{7}{8}$};
    \draw (0,4) -- (2,0) ; 
    \draw[dashed] (0, 3) -- (1.33, 1.33) ; 
    \draw[fill] (1.33,1.33) circle [radius=0.025]; 
    \node[right] at (1.33, 1.33) {$(\frac{2}{3}, \frac{2}{3})$} ;
\end{tikzpicture}
\caption{The range of $1 \les q, r\les 2$ in $d=3$. The line corresponds to the sharp bilinear line $\frac{2}{q} + \frac{d+1}{r} = d+1$ given by Theorem \ref{thm:wave-schrodinger bilinear free solns}. If \eqref{eqn:wave-schro trans} holds but \eqref{eq:wave-schro general trans} fails, then Theorem \ref{thm:transverse counter} states that the bilinear estimate \eqref{eqn:bilinear intro} can only hold to the left of the dotted line. }
\label{fig:range of q r}
\end{figure}
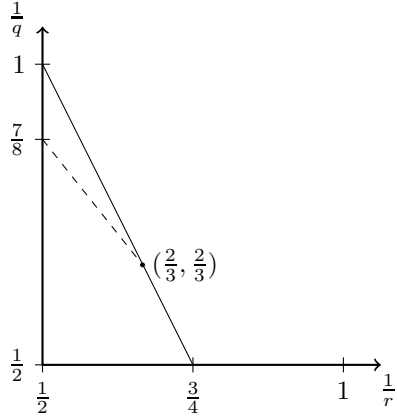

\begin{theorem}[Non-transverse counter example]\label{thm:non-transverse counter}
Suppose that the estimate \eqref{eqn:bilinear intro} holds  for all $f, g \in L^2(\RR^d)$ with $\supp \widehat{f}, \supp \widehat{g} \subset \{ |\xi| \approx 1\}$. Then
        \begin{equation}\label{eqn:counter example range nontransverse}
            \frac{2}{q} + \frac{d-1}{r} \les d - \frac{1}{2}, \qquad \frac{1}{q} \les \frac{d+1}{4}.
        \end{equation}
\end{theorem}

We give the proof of Theorem \ref{thm:transverse counter} and Theorem \ref{thm:non-transverse counter} in Section \ref{sec:counter} below. In the positive direction, if \eqref{eqn:wave-schro trans} holds, but \eqref{eq:wave-schro general trans} fails, a naive adaption of the proof of Theorem \ref{thm:wave-schrodinger bilinear free solns} should give the region $\frac{2}{q} + \frac{d}{r} < d$. The loss of dimension corresponds to the lack of curvature in the radial direction (i.e. the cone only has $d-1$ directions of non-vanishing curvature). Note that there is a large gap between the potential  range  $\frac{2}{q} + \frac{d}{r} < d$ and the counter example given by Theorem \ref{thm:transverse counter}. Similarly, even in the ``linear'' case when the transversality condition \eqref{eqn:wave-schro trans} is dropped, there is a gap between the counter example in Theorem \ref{thm:non-transverse counter} and the linear range given via Strichartz estimates \eqref{eqn:bi range via stric}. It is an interesting open question to determine the precise range of $(q,r)$ once the general condition \eqref{eq:wave-schro general trans} is dropped. In particular, it is not clear to the author what the optimal range for $(q, r)$ should be. Presumably the counter examples used in the proof of Theorem \ref{thm:transverse counter} and Theorem \ref{thm:non-transverse counter} can be improved. \\

In applications to nonlinear PDE, typically the homogeneous estimate in Theorem \ref{thm:wave-schrodinger bilinear free solns} is not sufficient, and it is more useful to have a version in suitable function spaces. One option is to work with $X^{s,b}$ type spaces. However, recently bilinear restriction estimates in the $U^p$ type spaces have proven useful, see for instance \cite{Candy2018c} and the discussion within. In the following we wish to give a general argument which can allow multilinear estimates for homogeneous solutions, to be upgraded to estimates in the adapted function spaces $U^2$. The underlying idea is straight forward. The first step is use the classical theorem of Marcinkiewicz-Zygmund that given a bound for a linear operator, a standard randomisation argument via Khintchine's inequality implies that a vector valued operator bound also holds. The second step is use the observation that a vector valued estimate immediately implies a $U^2$ bound, see for instance \cite[Section 1.2]{Candy2019b} or \cite[Remark 5.2]{Candy2018a}. As an example, we extend Theorem \ref{thm:wave-schrodinger bilinear free solns}, and the multilinear restriction theorem \cite{Bennett2006} to $U^2$.

We start with the definition of $U^2$. A function $\phi \in L^\infty_t L^2_x$ is an $\emph{atom}$ if we can write $ \phi(t) = \sum_{I} \ind_I(t) g_I$, with the intervals $I \subset \RR$ forming a partition of $\RR$, and the $g_I:\RR^d \to \CC$ satisfying the bound
        $$\Big( \sum_I \| g_I \|_{L^2_x}^2\Big)^\frac{1}{2} \les 1.$$
The atomic space $U^2$ is then defined as
        $$U^2= \Big\{ \sum_j c_j \phi_j \mid \text{ $\phi_j$ an atom and } (c_j) \in \ell^1 \Big\} $$
with the induced norm
        $$ \| u \|_{U^2} = \inf_{ u = \sum_j c_j \phi_j} \sum_j |c_j|$$
where the inf is over all representations of $u$ in terms of atoms. These spaces were introduced in unpublished work of Tataru, and studied in detail in \cite{Koch2005, Hadac2009}. To obtain the adapted function spaces $U^2_{|\nabla|}$ and $U^2_{\Delta}$ adapted to the wave and Schr\"odinger flows respectively, we define
    $$  U^2_{|\nabla|} = \{ u: \RR^{1+d} \to \CC \mid e^{-it|\nabla|} u \in U^2 \}, \qquad U^2_{\Delta} = \{ v : \RR^{1+d} \to \CC \mid e^{-it\Delta} v \in U^2 \}. $$
Note that since $\ind_\RR(t) f \in U^2$, we clearly have $e^{it|\nabla|} f \in U^2_{|\nabla|}$ and $e^{it\Delta} f \in U^2_{\Delta}$. Thus the adapted function spaces contain all homogeneous solutions. Running the argument sketched above implies the following $U^2$ version of Theorem \ref{thm:wave-schrodinger bilinear free solns}.

\begin{theorem}[Wave-Schr\"odinger bilinear restriction in $U^2$]\label{thm:wave-schro bi U2}
Let $d\g 2$, $1\les q, r \les 2$, and $\frac{2}{q} + \frac{d+1}{r} < d+1$. Let $\xi_0, \eta_0 \in \RR^d$ such that
    \begin{equation}\label{eqn:wave-schro general trans U2} \Big| \Big( \frac{\xi_0}{|\xi_0|} + 2 \eta_0 \Big) \cdot \frac{\xi_0}{|\xi_0|}\Big| \gtrsim \Big| \frac{\xi_0}{|\xi_0|} + 2 \eta_0 \Big| \end{equation}
and define $\lambda =  |\eta_0|$, and $\alpha = | \frac{\xi_0}{|\xi_0|} + 2 \eta_0 |$. If
	$$\supp \widehat{u} \subset \big\{ |\xi| \approx \lambda, \ma(\xi, \xi_0) \ll  \min\{1, \alpha \} \big\}, \qquad \supp \widehat{v} \subset \{ |\xi - \eta_0| \ll \alpha\}$$
then we have
    $$\|  u v  \|_{L^q_t L^r_x(\RR^{1+d})} \lesa ( \min\{\alpha, \lambda, \alpha \lambda\})^{d+1 - \frac{d+1}{r} - \frac{2}{q}} \alpha^{\frac{1}{r}-1} \lambda^{\frac{1}{q} - \frac{1}{2}} \|u\|_{U^2_{|\nabla|}} \| v \|_{U^2_{\Delta}}. $$
\end{theorem}
\begin{proof}
Let $u = \sum_{I\in \mc{I}} e^{it|\nabla|} f_I$ be a $U^2_{|\nabla|}$ atom, and let $v_0 = e^{it\Delta} g$ be a homogeneous solution to the Schr\"odinger equation. Assume that the support conditions \eqref{eqn:wave-schro general trans U2} hold. Let $(\epsilon_I)_{I\in \mc{I}}$ be a family of independent, identically distributed random variables with $\epsilon_I = 1$ with probability $\frac{1}{2}$, and $\epsilon_I = -1$ with probability $\frac{1}{2}$. The since the intervals $I$ are disjoint, we have via Khintchine's inequality
        $$ |u| \les \Big( \sum_I | e^{it|\nabla|} f_I |^2 \Big)^\frac{1}{2} \approx \mb{E}\Big[ \Big|\sum_I \epsilon_I e^{it|\nabla|} f_I\Big| \Big]. $$
Therefore, since $q, r \g 1$, applying Theorem \ref{thm:wave-schrodinger bilinear free solns} gives
        \begin{align*}
             \| u v_0 \|_{L^q_t L^r_x} &\lesa \Big\| \mb{E}\Big[ \Big|\sum_I \epsilon_I e^{it|\nabla|} f_I\Big| \Big] v_0 \Big\|_{L^q_t L^r_x} \\
                                       &\lesa \mb{E}\Big[ \Big\|    e^{it|\nabla|} \Big( \sum_I \epsilon_I f_I\Big)  v_0 \Big\|_{L^q_t L^r_x} \Big]\\
                                       &\lesa  ( \min\{\alpha, \lambda, \alpha \lambda\})^{d+1 - \frac{d+1}{r} - \frac{2}{q}} \alpha^{\frac{1}{r}-1} \lambda^{\frac{1}{q} - \frac{1}{2}}
                                                \mb{E}\Big[ \Big\| \sum_I \epsilon_I f_I \Big\|_{L^2_x} \Big] \| g \|_{L^2_x}.
        \end{align*}
We now observe that H\"older's inequality, together with another application of Khintchine's inequality, implies that
$$  \mb{E}\Big[ \Big\| \sum_I \epsilon_I f_I \Big\|_{L^2_x} \Big] \les  \Big(\mb{E}\Big[ \Big\| \sum_I \epsilon_I f_I \Big\|_{L^2_x}^2 \Big]\Big)^\frac{1}{2} = \Big( \sum_I \|f_I\|_{L^2}^2 \Big)^\frac{1}{2}$$
and consequently, applying the definition of the $U^2_{|\nabla|}$ norm, we obtain
        \begin{equation}\label{eqn:thm wave-schro U2: one hom}
            \| u v_0 \|_{L^q_t L^r_x} \lesa ( \min\{\alpha, \lambda, \alpha \lambda\})^{d+1 - \frac{d+1}{r} - \frac{2}{q}} \alpha^{\frac{1}{r}-1} \lambda^{\frac{1}{q} - \frac{1}{2}}
                                                \| u \|_{U^2_{|\nabla|}} \| g \|_{L^2_x}.
        \end{equation}
To replace the homogeneous solution $v_0$ with a general $U^2_\Delta$ function follows by essentially repeating the above argument. In slightly more detail, suppose that $v = \sum_{J \in \mc{J}} e^{it\Delta} g_J$ is a $U^2_{\Delta}$ atom, and let $(\epsilon_J)_{J\in \mc{J}}$ be a family of i.i.d. random variables with $\epsilon_J = \pm 1$ with equal probability. Then as above, but replacing Theorem \ref{thm:wave-schrodinger bilinear free solns} with \eqref{eqn:thm wave-schro U2: one hom}, we see that
    \begin{align*}
      \| u v \|_{L^q_t L^r_x} &\lesa \Big\| u \mb{E}\Big[ \Big| \sum_J \epsilon_J e^{it\Delta} g_J \Big| \Big] \Big\|_{L^q_t L^r_x} \\
                              &\lesa  \mb{E}\Big[  \Big\| u \sum_J e^{it\Delta} \epsilon_J g_J \Big\|_{L^q_t L^r_x} \Big] \\
                              &\lesa  ( \min\{\alpha, \lambda, \alpha \lambda\})^{d+1 - \frac{d+1}{r} - \frac{2}{q}} \alpha^{\frac{1}{r}-1} \lambda^{\frac{1}{q} - \frac{1}{2}}
                                                \| u \|_{U^2_{|\nabla|}}  \mb{E}\Big[ \Big\| \sum_J \epsilon_J g_J \Big\|_{L^2}\Big]\\
                              &\lesa ( \min\{\alpha, \lambda, \alpha \lambda\})^{d+1 - \frac{d+1}{r} - \frac{2}{q}} \alpha^{\frac{1}{r}-1} \lambda^{\frac{1}{q} - \frac{1}{2}}
                                                \| u \|_{U^2_{|\nabla|}} \Big(\sum_J \| g_J \|_{L^2}^2 \Big)^\frac{1}{2}.
    \end{align*}
Applying the definition the $U^2_\Delta$ norm, the required bound follows.
\end{proof}

Strictly speaking the above theorem can also be obtain via the vector valued version of Theorem \ref{thm:wave-schrodinger bilinear free solns} from \cite{Candy2019b}, see for instance \cite[Section 1.2]{Candy2019b}. However the above alternative argument is more direct, and has the distinct advantage that it can be applied in more general situations. As an example, consider the following special case of the multilinear restriction theorem  \cite{Bennett2006}.

\begin{theorem}[Multilinear restriction for Schr\"odinger { \cite{Bennett2006}} ]\label{thm:multi restr}
Let $d\g 2$ and $\epsilon>0$. Then for any $R\g 1$ and any $f_j \in L^2(\RR^d)$, $j=1, \dots, d$  with $ \supp \widehat{f_j} \subset \{ |\xi - e_j| \ll 1 \}$ we have
            $$ \Big\| \Pi_j e^{it\Delta} f_j \Big\|_{L^{\frac{2}{d-1}}_{t,x}( \{ |t|+|x| < R\})} \lesa R^\epsilon \Pi_j \| f_j \|_{L^2}. $$
\end{theorem}

It is conjectured that the $R^\epsilon$ loss can be removed, but this is currently an open question (however see \cite{Bejenaru2020} and \cite{Bejenaru2019a} for recent progress). The $U^2$ version of Theorem \ref{thm:multi restr} is then the following.

\begin{theorem}[Multilinear restriction for Schr\"odinger in $U^2$]\label{thm:multi restr U^2}
Let $d\g 2$ and $\epsilon>0$. Then for any $R\g 1$ and any $u_j \in U^2_{\Delta}$, $j=1, \dots, d$  with $ \supp \widehat{u} \subset \{ |\xi - e_j| \ll 1 \}$ we have
            $$ \Big\| \Pi_j u_j \Big\|_{L^{\frac{2}{d-1}}_{t,x}( \{ |t|+|x| < R\})} \lesa R^\epsilon \Pi_j \| u_j \|_{U^2_\Delta}. $$
\end{theorem}
\begin{proof}
Let $B_R = \{ |t|+|x| < R\}$. We proceed as in the proof of Theorem \ref{thm:wave-schro bi U2}. Thus suppose that $u_1 = \sum_I e^{it\Delta} f_I$ is $U^2_\Delta$ atom, and let $u_j^0 = e^{it\Delta} f_j$ for $j=2, \dots, d$. Let $\epsilon_I$ be a family of i.i.d. random variables with $\epsilon_I = \pm 1$ with equal probability. An application of Khintchine's inequality implies that
   $$ |u_1| \les \Big( \sum_I | e^{it\Delta} f_I |^2 \Big)^\frac{1}{2} \approx \Big(\mb{E}\Big[ \Big|\sum_I \epsilon_I e^{it|\nabla|} f_I\Big|^{\frac{2}{d-1}} \Big]\Big)^{\frac{d-1}{2}} $$
and hence Theorem \ref{thm:multi restr} together with H\"older's inequality gives
   \begin{align*}
         \Big\| u_1 \Pi_{j=2}^d u_j^0 \Big\|_{L^{\frac{2}{d-1}}_{t,x}(B_R)} &\lesa \Big\| \Big(\mb{E}\Big[ \Big|\sum_I \epsilon_I e^{it\Delta} f_I\Big|^{\frac{2}{d-1}} \Big]\Big)^\frac{d-1}{2} \Pi_{j=2}^d u_j^0 \Big\|_{L^{\frac{2}{d-1}}_{t,x}(B_R)} \\
         &\lesa \Big(\mb{E}\Big[ \Big\| \sum_I \epsilon_I e^{it\Delta} f_I \Pi_{j=2}^d u_j \Big\|_{L^{\frac{2}{d-1}}_{t,x}(B_R)}^{\frac{2}{d-1}} \Big]\Big)^{\frac{d-1}{2}} \\
         &\lesa R^\epsilon \Big(\mb{E}\Big[ \Big\| \sum_I \epsilon_I f_I \Big\|_{L^2}^{\frac{2}{d-1}} \Big]\Big)^\frac{d-1}{2} \Pi_{j=2}^d \| f_j \|_{L^2} \\
         &\lesa R^\epsilon  \Big(\mb{E}\Big[ \Big\| \sum_I \epsilon_I f_I \Big\|_{L^2}^2 \Big]\Big)^\frac{1}{2} \Pi_{j=2}^d \| f_j \|_{L^2}  \approx R^\epsilon \Big( \sum_I \|f_I \|_{L^2}^2 \Big)^\frac{1}{2} \Pi_{j=2}^d \| f_j \|_{L^2}.
   \end{align*}
Applying the definition of the $U^2_\Delta$ norm, we conclude that
        \begin{equation}\label{eqn:thm multi restric U2:one U2}
            \Big\| u_1 \Pi_{j=2}^d u_j^0 \Big\|_{L^{\frac{2}{d-1}}_{t,x}(B_R)} \lesa R^\epsilon \| u_1 \|_{U^2_\Delta} \Pi_{j=2}^d \| f_j \|_{L^2}.
        \end{equation}
As in the proof of Theorem \ref{thm:wave-schro bi U2}, repeating this argument with Theorem \ref{thm:multi restr} replaced with \eqref{eqn:thm multi restric U2:one U2} and $u_1$ replaced with $u_2$ gives
         $$   \Big\| u_1 u_2 \Pi_{j=3}^d u_j^0 \Big\|_{L^{\frac{2}{d-1}}_{t,x}(B_R)} \lesa R^\epsilon \| u_1 \|_{U^2_\Delta} \| u_2 \|_{U^2_\Delta}  \Pi_{j=3}^d \| f_j \|_{L^2}. $$
The required bound follows by continuing in this manner.
\end{proof}

We have not attempted to write down the most general transference type argument that can be deduced from the above arguments. However the underlying idea is simple; if a estimate holds for free solutions, then via randomisation it should hold in the vector valued case, and consequently it will also hold in $U^2$. Of course proving $U^p$ bounds, with $p\not = 2$ is substantially more challenging.

\subsection*{Acknowledgements} The author would like to thank Sebastian Herr and Kenji Nakanishi for a number of helpful discussions, as well as the University of Bielefeld and MATRIX for their kind hospitality while part of this work was conducted.

\section{Proof of Theorem \ref{thm:wave-schrodinger bilinear free solns}}\label{sec:proof of bilinear}
It suffices to check  the conditions in \cite{Candy2019b}. Suppose that $\xi_0, \eta_0 \in \RR^d$ such that \eqref{eq:wave-schro general trans} holds, and define $\lambda =  |\eta_0|$, and $\alpha = | \frac{\xi_0}{|\xi_0|} + 2 \eta_0 |$.  Let
	$$\Lambda_1 = \big\{ |\xi| \approx \lambda, \ma(\xi, \xi_0) \ll  \min\{1, \alpha \} \big\}, \qquad \Lambda_2 =  \{ |\xi - \eta_0| \ll \alpha\}$$
and
    $$\Phi_1(\xi) = |\xi|, \qquad \Phi_2(\xi) = -|\xi|^2, \qquad \mc{H}_1 = \lambda^{-1}, \qquad \mc{H}_1= 1. $$
In view of \cite[Lemma 2.1 and Theorem 1.2]{Candy2019b}, for $\{j, k\}=\{1,2\}$ and $\xi \in \Lambda_j$, $\eta \in \Lambda_k$, it suffices to check the following conditions:
\begin{enumerate}
	\item for all $v\in \RR^d$ we have
		$$v \cdot ( \nabla \Phi_j(\xi) - \nabla \Phi_k(\eta) ) = 0 \qquad \Longrightarrow \qquad \big| \nabla^2 \Phi_j(\xi) v \wedge \big( \nabla \Phi_j(\xi) - \nabla \Phi_k(\eta) \big) \big| \gtrsim \mc{H}_j \alpha |v|, $$
		
	\item for $\xi' \in \Lambda_j$ and $\eta' \in \Lambda_k$ we have
			$$ | \nabla \Phi_j(\xi) - \nabla \Phi_j(\xi')| + |\nabla \Phi_k(\eta) - \nabla \Phi_k(\eta')| \ll \alpha,$$
			
	\item the Hessian's satisfy
			$$ | \nabla \Phi_j(\xi) - \nabla \Phi_j(\xi') - \nabla^2 \Phi_j(\xi) (\xi - \xi') | \ll \mc{H}_j |\xi - \xi'|,$$
			
	\item for $2<m \les 5d$ we have the derivative bounds
			$$ \| \nabla^m \Phi_j \|_{L^\infty(\Lambda_j)} (\min\{\alpha, \lambda, \alpha \lambda \})^{m-2} \lesa \mc{H}_j, \qquad \mc{H}_j \min\{\alpha, \lambda, \alpha \lambda\} \lesa \alpha. $$			
			
    \item we have the surface measure condition
    		$$ \sup_{(a,h) \in \RR^{1+d}} \sigma_{d-1}\big( \big\{ \xi \in \Lambda_2 \cap (h-\Lambda_1) \big| \,\, \Phi_2(\xi) + \Phi_1(h-\xi) = a\}  \big) \lesa \big( \min\{\alpha, \lambda, \alpha \lambda \} \big)^{d-1}$$
    	where $\sigma_{d-1}$ is the induced Lebesgue surface measure.
\end{enumerate}
To check the first property (i), by unpacking the definition, our goal is to show that for any $\xi \in \Lambda_1$ and $\eta \in \Lambda_2$ we have
    $$z \cdot ( \omega + 2\eta ) = 0 \qquad \Longrightarrow \qquad \big| (z - (\omega\cdot z) \omega )  \wedge ( \omega  + 2\eta) \big| \gtrsim |z| |\omega + 2\eta|, $$
where $\omega = \frac{\xi}{|\xi|}$. In view of the definition of the sets $\Lambda_j$ we have
       $$ | ( \omega + 2 \eta ) \cdot \omega | \gtrsim | \omega  + 2 \eta|$$
and hence as $ (z \cdot \omega) ( \omega + 2\eta) \cdot \omega = 2  z \cdot ( \eta - (\eta\cdot \omega) \omega) $ we get
	$$ |z \cdot \omega| \les \frac{2| z \cdot( \eta - (\eta \cdot \omega) \omega)| }{| ( \omega + 2 \eta) \cdot \omega| }  \lesa |z - (\omega \cdot z) \omega | .$$
Therefore
	\begin{align*}
		\big| \big( z -  (\omega \cdot z) \omega \big) \wedge ( \omega + 2\eta) \big| &\g \big| z -  (\omega \cdot z) \omega \big| |( \omega + 2\eta)\cdot \omega | \gtrsim |z| | \omega + 2\omega|
	\end{align*}
as required.

The properties  (ii), ... , (iv) follow by direct computation. Finally, to check the surface measure condition (v), we note that the vector $N=\frac{\xi_0}{|\xi_0|} + 2 \eta_0$ is essentially normal to the surface. On the other hand, from \eqref{eq:wave-schro general trans}, $N$ is roughly pointing in the direction $\frac{\eta_0}{|\eta_0|}$. Hence the surface measure can be bounded by projecting onto the plane orthogonal to $\frac{\eta_0}{|\eta_0|}$. Since this projection is contained in a ball of radius  $\lesa \min\{\alpha, \lambda, \alpha \lambda\}$, the bound follows.

\section{Counter Examples}\label{sec:counter}
We first observe that by a randomisation argument, if the estimate \eqref{eqn:bilinear intro} holds for all $f, g \in L^2$ with $\supp \widehat{f} \subset \Lambda_1 \subset \RR^n$ and $\supp \widehat{g} \subset \Lambda_2$, then in fact we also have the vector valued version
        \begin{equation}\label{eqn:bilinear vec val}
            \Big\| \Big( \sum_{j} |e^{it|\nabla|} f_j |^2 \Big)^\frac{1}{2} \Big( \sum_{k} |e^{it\Delta} g_k |^2 \Big)^\frac{1}{2} \Big\|_{L^q_t L^r_x} \lesa \Big( \sum_j \| f_j \|_{L^2}^2\Big)^\frac{1}{2} \Big( \sum_k \| g_k \|_{L^2_x}^2 \Big)^\frac{1}{2}
        \end{equation}
for all $\supp \widehat{f}_j \subset \Lambda_1$, $\supp \widehat{g}_k \subset \Lambda_2$. This follows by noting that if $\epsilon_j$ is an i.i.d. family of random variables with $\epsilon_j = \pm$ with equal probability, then as in the proof of Theorem \ref{thm:wave-schro bi U2}, we have via Khintchine's inequality and \eqref{eqn:bilinear intro}
    \begin{align*}
      \Big\| \Big( \sum_j |e^{it|\nabla|} f_j |^2 \Big)^\frac{1}{2} e^{it\Delta } g \Big\|_{L^q_t L^r_x} &\approx \Big\| \mb{E} \Big[ \Big| \sum_{j} \epsilon_j e^{it|\nabla|} f_j \Big| \Big] e^{it\Delta} g \Big\|_{L^q_t L^r_x} \\
      &\lesa \mb{E}\Big[  \Big\| \sum_{j} \epsilon_j e^{it|\nabla|} f_j e^{it\Delta} g \Big\|_{L^q_t L^r_x} \Big] \\
      &\lesa \mb{E}\Big[ \Big\| \sum_j \epsilon_j e^{it|\nabla|} f_j \Big\|_{L^2_x}\Big] \| g \|_{L^2_x} \\
      &\lesa \mb{E}\Big[ \Big\| \sum_j \epsilon_j e^{it|\nabla|} f_j \Big\|_{L^2_x}^2\Big]^\frac{1}{2} \| g \|_{L^2_x}\approx \Big( \sum_j \| f_j \|_{L^2}^2 \Big)^\frac{1}{2} \| g\|_{L^2}.
    \end{align*}
Repeating this argument for the Schr\"odinger component then gives \eqref{eqn:bilinear vec val}. Consequently, we see that the scalar version \eqref{eqn:bilinear intro} holds, if and only if the vector valued version \eqref{eqn:bilinear vec val} holds. Thus to prove Theorem \ref{thm:transverse counter} and Theorem \ref{thm:non-transverse counter}, it suffices to obtain vector valued counter examples.

\subsection{Proof of Theorem \ref{thm:transverse counter} } Let $N\g 1$ and $\widehat{f}, \widehat{g} \in C^\infty_0$ with
        $$ \supp \widehat{f} \subset \{ |\xi_1 - 1| \ll 1, |\xi'| \ll N^{-1} \}, \qquad \supp \widehat{g} \subset \{ |2 \xi - e_1 - e_2 | \ll N^{-\frac{1}{2}} \}$$
and $\| f \|_{L^2} \approx N^{\frac{d-1}{2}}$, $\| g \|_{L^2} \approx N^{\frac{d}{4}}$. A short computation using integration by parts shows that we can choose $f, g$ such that
        $$ |u(t,x)| = |e^{it|\nabla|} f(x) | \g 1 \qquad \text{for all } \qquad |t|\les N^{2}, \,\, |x_1 + t| \les 1, \,\, |x'| \les N $$
and
        $$ |v(t,x)| = |e^{it\Delta} g(x) | \g 1 \qquad \text{ for all } \qquad |t|\les N^{1}, \,\, |x_1 + t| \les N^\frac{1}{2}, \,\, |x_2 + t | \les N^{\frac{1}{2}}, \,\, |x''| \les N^{\frac{1}{2}}$$
where we write $x = (x_1, x') = (x_1, x_2, x'')$. In other words the free wave $u$ is $\g 1$ on a plate of dimension $N^2 \times 1 \times N^{d-1}$ oriented in the $(1, -e_1)$ direction, with short direction $e_1$, while the free Schr\"odinger wave $v$ is $\g 1 $ on a tube of dimensions $N \times N^{\frac{d}{2}}$ oriented in the $(1, -e_1 - e_2)$ direction. Define the set
        $$ \Omega = \{ |t| \les N^{-2}, |x_1 + t|\les N^\frac{1}{2}, |x'|\les N \}. $$
The support properties of $u$, implies that for any $(t,x) \in \Omega$ we have
        $$ U(t,x) = \Big( \sum_{\substack{j \in \ZZ \\ |j|\les N^{\frac{1}{2}}}} |u(t, x + j e_1 )|^2 \Big)^\frac{1}{2} \gtrsim 1. $$
Similarly, translating the free Schr\"odinger wave in both space and time gives for any $(t, x) \in \Omega$
        $$ V(t,x) =  \Big( \sum_{\substack{(j_2, \dots, j_d)\in \ZZ^{d-1} \\ |j_1|, \dots, |j_d| \lesa N^{\frac{1}{2}} }} \sum_{\substack{k\in \ZZ \\ |k|\les N }} \big|v\big(t + N k, x + N^\frac{1}{2}(j_2 e_2 + \dots + j_d e_d) \big)\big|^2 \Big)^\frac{1}{2} \gtrsim 1. $$
Since the wave and Schr\"odinger equations are translation invariant, the bound \eqref{eqn:bilinear vec val} implies that
  \begin{align*}
    N^{\frac{2}{q}} N^{\frac{d-1}{r} + \frac{1}{2r}} \lesa \| \ind_\Omega \|_{L^q_t L^r_x} &\lesa \| U V \|_{L^q_t L^r_x} \\
                &\lesa \Big( \sum_{|j| \les N^\frac{1}{2}} \| f \|_{L^2}\Big)^\frac{1}{2} \Big( \sum_{ |j_1|, \dots, |j_d| \lesa N^{\frac{1}{2}}}  \sum_{|k|\les N } \| g \|_{L^2}^2 \Big)^\frac{1}{2} \lesa N^{\frac{d-1}{2} + \frac{1}{4}} \times N^{\frac{d}{2} + \frac{1}{4}}.
  \end{align*}
Letting $N \to \infty$, we see that this is only possible if
        $$ \frac{2}{q} + \frac{d-1}{r} + \frac{1}{2r} \les d. $$

\subsection{Proof of Theorem \ref{thm:non-transverse counter}} Let $1\les M \les N$ and $\widehat{f}, \widehat{g} \in C^\infty_0$ with
        $$ \supp \widehat{f} \subset \{ |\xi_1 - 1| \ll 1, |\xi'| \ll N^{-1} \}, \qquad \supp \widehat{g} \subset \{ |2 \xi - e_1| \ll M^{-1} \}$$
and $\| f \|_{L^2} \approx N^{\frac{d-1}{2}}$, $\| g \|_{L^2} \approx M^{\frac{d}{2}}$. A short computation using integration by parts shows that we can choose $f, g$ such that
        $$ |u(t,x)| = |e^{it|\nabla|} f(x) | \g 1 \qquad \text{for all } \qquad |t|\les N^{2}, \,\, |x_1 + t| \les 1, \,\, |x'| \les N $$
and
        $$ |v(t,x)| = |e^{it\Delta} g(x) | \g 1 \qquad \text{ for all } \qquad |t|\les M^{1}, \,\, |x_1 + t| \les M, \,\, |x'| \les M.$$
In other words the free wave $u$ is $\g 1$ on a plate of dimension $N^2 \times 1 \times N^{d-1}$ oriented in the $(1, -e_1)$ direction, with short direction $e_1$, while the free Schr\"odinger wave $v$ is $\g 1 $ on a tube of dimensions $M^2 \times M^d$ oriented in the $(1, -e_1)$ direction. Similar to the proof of Theorem \eqref{thm:transverse counter}, we consider a number of temporal translated Schr\"odinger waves covering the set
        $$ \Omega = \{ |t| \les N^{-2}, |x_1 + t|\les 1, |x'|\les M \}. $$
More precisely, we have for all $(t,x) \in \Omega$
            $$V(t,x) = \Big( \sum_{\substack{ j \in \ZZ \\ |j| \les \frac{N^2}{M^2}}} |v(t+j M^2, x)|^2 \Big)^\frac{1}{2} \gtrsim 1 .$$
Since we clearly have $|u| \g 1$ on $\Omega$ by construction, we see that if \eqref{eqn:bilinear intro} holds, then the vector valued version \eqref{eqn:bilinear vec val} holds, and hence
    $$ N^{\frac{2}{q}} M^{\frac{d-1}{r}} \lesa \| \ind_\Omega \|_{L^q_t L^r_x} \lesa \| u V \|_{L^q_t L^r_x} \lesa \| f \|_{L^2}  \Big( \sum_{ |j| \les \frac{N^2}{M^2}} \|g\|_{L^2}^2 \Big)^\frac{1}{2} \lesa N^{\frac{d-1}{2}} \times N M^{\frac{d}{2}-1}. $$
Rearranging, we see that we must have
    $$N^{\frac{2}{q} - \frac{d+1}{2}} M^{\frac{d-1}{r} - \frac{d-2}{2} } \lesa 1.$$
Letting $M=1$ and $N\to \infty$ gives the restriction
        $$ \frac{1}{q} \les \frac{d+1}{4}. $$
On the other hand, letting $M=N \to \infty$, gives the condition
        $$ \frac{2}{q} + \frac{d-1}{r} \les d -\frac{1}{2}. $$

\bibliographystyle{amsplain}
\providecommand{\bysame}{\leavevmode\hbox to3em{\hrulefill}\thinspace}
\providecommand{\MR}{\relax\ifhmode\unskip\space\fi MR }
\providecommand{\MRhref}[2]{%
  \href{http://www.ams.org/mathscinet-getitem?mr=#1}{#2}
}
\providecommand{\href}[2]{#2}

\end{document}